\documentclass{amsart}
\usepackage{graphicx, amssymb, amsmath, amsfonts, mathtools, }
\usepackage{tikz, tikz-cd,xcolor}

\usepackage[numbers,sort&compress]{natbib}
\usepackage{hyperref}
\usepackage{aliascnt}
\usepackage[capitalise]{cleveref}
\usepackage{doi}
\usepackage{todonotes}
\newcommand{\rips}{\mathcal{VR}}
\DeclareMathOperator{\Cl}{Cl}
\DeclareMathOperator{\diam}{diam}
\newcommand{\field}{\Bbbk}

\newcommand{\Top}{\mathbf{Top}}
\newcommand{\Simp}{\mathbf{SCpx}}
\newcommand{\Vect}{\mathbf{Vec}}

\newtheorem{theorem}{Theorem}[section]
\crefname{theorem}{theorem}{theorems}
\Crefname{theorem}{Theorem}{Theorems}

\newaliascnt{lemma}{theorem}
\newtheorem{lemma}[lemma]{Lemma}
\aliascntresetthe{lemma}
\crefname{lemma}{lemma}{lemmas}
\Crefname{lemma}{Lemma}{Lemmas}

\newaliascnt{proposition}{theorem}
\newtheorem{proposition}[proposition]{Proposition}
\aliascntresetthe{proposition}
\crefname{proposition}{proposition}{propositions}
\Crefname{proposition}{Proposition}{Propositions}

\newaliascnt{corollary}{theorem}
\newtheorem{corollary}[corollary]{Corollary}
\aliascntresetthe{corollary}
\crefname{corollary}{corollary}{corollaries}
\Crefname{corollary}{Corollary}{Corollaries}

\newaliascnt{conjecture}{theorem}

\aliascntresetthe{conjecture}
\crefname{conjecture}{conjecture}{conjectures}
\Crefname{conjecture}{Conjecture}{Conjectures}

\theoremstyle{definition}
\newaliascnt{definition}{theorem}
\newtheorem{definition}[definition]{Definition}
\aliascntresetthe{definition}
\crefname{definition}{definition}{definitions}
\Crefname{definition}{Definition}{Definitions}

\newaliascnt{remark}{theorem}
\newtheorem{remark}[remark]{Remark}
\aliascntresetthe{remark}
\crefname{remark}{remark}{remarks}
\Crefname{remark}{Remark}{Remarks}

\newaliascnt{question}{theorem}

\aliascntresetthe{question}
\crefname{question}{question}{questions}
\Crefname{question}{Question}{Questions}

\title{Lower Bounds for Approximating the Vietoris-Rips Filtration}

\author{Kenneth McCabe}
\address{Department of Mathematics, Northeastern University}
\email{mccabe.ke@northeastern.edu}

\begin{document}

\begin{abstract}
The Vietoris--Rips filtration $\rips(-)$ is a standard tool for analyzing the shape of data within topological data analysis. Beginning with seminal work of Sheehy, a substantial amount of research has centered on constructing linear-size sparse approximations to $\rips(-)$ and related filtrations for metric spaces of bounded doubling dimension. We show that this geometric assumption is necessary in a precise sense. Working in the framework of homotopy interleavings, we show that for any fixed $c \in [1, \sqrt{2})$, there exists a family of finite metric spaces for which any finitely presented $c$-approximation to $\rips(-)$ has exponential size. We also show that for any fixed $c \geq 1$, there exists a family of finite metric spaces for which any finitely presented $c$-approximation to $\rips(-)$ has superlinear size, yielding an obstruction to linear-size approximations for any fixed approximation factor. Both results extend to the intrinsic \v{C}ech filtration and to any bifiltration containing $\rips(-)$ as a $1$-parameter slice, including the function-Rips, degree-Rips, and subdivision-Rips bifiltrations.
\end{abstract}

\maketitle

\section{Introduction}

For a finite metric space $X$ and scale $r \geq 0$, the \emph{Vietoris-Rips complex} $\rips(X)_r$ is the simplicial complex whose simplices are the nonempty subsets of $X$ of diameter at most $2r$. Letting $r$ vary gives the \emph{Vietoris-Rips filtration} $\rips(X)$. The Vietoris-Rips filtration is a central tool within topological data analysis, typically via persistent homology computations on metric data, and has been studied extensively in its own right. However, as the $k$-skeleton of $\rips(X)$ has $\Theta(|X|^{k+1})$ simplices, direct computations can become infeasible as the size of $X$ or the homological dimension of interest grows. This motivates the search for filtrations of asymptotically smaller size whose persistent homology closely approximates that of $\rips(-)$.

\subsection{Related Work}\label{Sec:Related-Work}
We use \emph{approximation} informally throughout this section to mean a construction whose persistent homology is close to that of the target; the precise notion varies by paper. Our approximation results use the language of \emph{homotopy interleavings} \cite{blumberg2023universality}, a notion that implies approximation at the level of persistent homology, but is not implied by it. By an \emph{exact model} of a filtration $\mathcal{F}$ we mean a functor valued in simplicial complexes that is weakly equivalent to $\mathcal{F}$; see \Cref{sec:Preliminaries} for precise definitions. 

\subsubsection{Sparse Filtrations}
 For finite metric spaces $X$ of bounded doubling dimension, Sheehy \cite{sheehyLinearSizeApproximations2013} showed that $\rips(X)$ admits a $(1+\epsilon)$-approximation of size $O(|X|)$ for any fixed $\epsilon > 0$. Botnan and Spreemann \cite{botnan_approximating_2015} subsequently extended this result to \v{C}ech filtrations of Euclidean point clouds, and further improvements and variants have been developed in \cite{cavanna15geometric, choudhary_improved_2021, choudhary_polynomial-sized_2019, sheehy_sparse_2021,dey_simba_2019, BuchetChazalOudotSheehy2016, ChoudharyKerberRaghvendra2019}. Complementing this line of work, Edelsbrunner et al. \cite{edelsbrunnerMaximumPersistent2026} showed via direct geometric arguments that for point clouds in $\mathbb R^d$, only a linear number of homological features in Vietoris-Rips and \v Cech filtrations can persist over an interval of fixed length.

Without geometric assumptions on $X$, the situation is more delicate. Choudhary et al. \cite{choudhary_polynomial-sized_2019} achieved an $O(\mathrm{polylog}(|X|))$-approximation to $\rips(X)$ of size $|X|^{O(1)}$ for arbitrary finite metric spaces $X$. The approximation factor, however, grows with $|X|$. Brun and Blaser \cite{brun_sparse_2019} defined a $(1+\epsilon)$-approximation to \v Cech filtrations of point clouds in arbitrary metric spaces, extending the construction of \cite{cavanna15geometric}, but did not give a formal size analysis. In a recent preprint, Leit\~ao \cite{leitao_its_2026} constructed a $3$-approximation to $\rips(-)$ for arbitrary metric spaces, but also did not prove any size bounds.

Analogous questions have been studied for multiparameter filtrations. The multicover bifiltration $\mathcal{M}(-)$, introduced by Sheehy \cite{sheehy12multicover} and studied further by Edelsbrunner and Osang \cite{edelsbrunner2021multicover} and Corbet et al.\ \cite{corbet2023computing}, is a density-sensitive \cite{blumberg2024stability} bifiltration for point sets $X$ in $\mathbb{R}^d$ admitting exact models of size $O(|X|^{d+1})$; no asymptotically smaller models are currently known. Buchet et al.\ \cite{buchetSparseHigher2023} introduced a linear-size sparse $(1+\epsilon)$-approximation to $\mathcal{M}(-)$ when the cover multiplicity $\mu$ is fixed, though the size bound depends exponentially on $\mu$. Lesnick and McCabe \cite{lesnick2024nerve, lesnickSparseApproximation2024} studied Sheehy's subdivision-Rips bifiltration $\mathcal{SR}(-)$ \cite{sheehy12multicover}, a density-sensitive refinement of $\rips(-)$ defined for arbitrary finite metric spaces. They constructed a $\sqrt{2}$-approximation to $\mathcal{SR}(-)$ for arbitrary metric spaces, and a $(1+\epsilon)$-approximation for metric spaces of bounded doubling dimension, both with polynomial-size skeleta. Hellmer and Spali\'nski \cite{hellmerDensitySensitive2024} independently gave a related construction via bifiltered Dowker complexes. Alonso \cite{alonso_sparse_2025} subsequently obtained a linear-size $(1+\epsilon)$-approximation to $\mathcal{M}(-)$ for fixed $d$ and $\epsilon$, and extended this result to $\mathcal{SR}(-)$ for metric spaces of bounded doubling dimension.

\subsubsection{Lower Bounds}
For exact models of $\rips(-)$, several lower bounds can be inferred from results on Betti numbers of Vietoris-Rips complexes: Goff \cite{goff_extremal_2011} constructed finite point sets $X$ in Euclidean space whose Vietoris--Rips complexes have $i^{\mathrm{th}}$ Betti number $\Omega(|X|^{i/2})$. Adams and Virk \cite{adams_lower_2024} obtained related lower bounds for Vietoris--Rips complexes of hypercube graphs. Adamaszek \cite{adamaszek_extremal_2014} exhibited an $n$-vertex flag complex whose total Betti number is $\Theta(4^{n/5})$, and proved that this is the maximum possible. Since any flag complex can be realized as the Vietoris-Rips complex of a suitable metric \cite[Appendix~A]{beers_extremal_2025_arXiv}, this implies that exact models of $\rips(-)$ have worst-case exponential size. Improving upon Adamaszek's result, Beers and Botnan \cite{beers_extremal_2025} showed that for any $i$, the $i^{\mathrm{th}}$-Betti number of a flag complex on $n$ vertices is maximized by the flag complex of the Tur\'an graph $T(n, i+1)$. When $5$ divides $n$, setting $i=n/5 -1$ recovers Adamaszek's example \cite[Example~6.1]{adamaszek_extremal_2014}. Tur\'an graphs will play a central role in our results as well; see \Cref{Thm:Main-Theorem} and \Cref{Rmk:Turan}.

In contrast, lower bounds on the size of approximations of $\rips(-)$ and related filtrations are less understood. Choudhary et al.\ \cite{choudhary_polynomial-sized_2019} proved a superpolynomial lower bound for the \v{C}ech filtration of Euclidean point clouds: for any fixed $\gamma \in (0,1)$, there exist finite $X \subseteq \mathbb R^d$ such that every $(1+\delta)$-approximation with $\delta < 1/\log^{1+\gamma}|X|$ contains $n^{\Omega(\log \log n)}$ intervals in its persistence barcode. The approximation factor, however, must shrink with $|X|$, and the authors noted their methods do not extend to $\rips(-)$. For the subdivision-Rips bifiltration, Lesnick and McCabe \cite{lesnick2024nerve} showed that exact models have exponential size for a large class of planar point sets, and that no $c$-approximation with polynomially-sized skeleta exists for arbitrary metric spaces when $c \in [1, \sqrt 2)$.

\subsection{Contributions}
We prove lower bounds on the size of $c$-approximations, in the sense of homotopy interleavings, to $\rips(-)$ and related filtrations. Our main results are as follows.
 
\begin{enumerate}
    \item For any fixed $c \in [1, \sqrt{2})$, there exists an infinite family of finite metric spaces for which any finitely presented $c$-approximation to $\rips(-)$ has exponential size (\Cref{Thm:Main-Theorem}). To the author's knowledge, this is the first explicit lower bound on the size of $c$-approximations for $\rips(-)$.
    \item For any fixed $c \geq 1$, there exists an infinite family of finite metric spaces for which any finitely presented $c$-approximation to $\rips(-)$ has superlinear size (\Cref{Thm:Superlinear-Bound}). This shows that the linear-size $(1+\epsilon)$-approximations available for metric spaces of bounded doubling dimension do not extend to arbitrary metric spaces for any fixed $\epsilon > 0$.
    \item Both results extend to the intrinsic \v{C}ech filtration and to any bifiltration containing $\rips(-)$ as a 1-parameter slice, including the function-Rips \cite{carlssonTheoryMultidimensional2009}, degree-Rips \cite{lesnickInteractiveVisualization2015a}, and subdivision-Rips bifiltrations (\Cref{Cor:Bifiltration-Exponential,Cor:Bifiltration-Polynomial}). In the case of subdivision-Rips, this extends a result of Lesnick and McCabe \cite[Corollary~1.8\,(ii)]{lesnick2024nerve}.
\end{enumerate}
 \sloppypar
Contributions $(1)$ and $(2)$ rely on the observation that any finitely presented $c$-approximation to a filtration $\mathcal{G}$ must be at least as large as the rank of the structure map $H_i(\mathcal{G})_r \to H_i(\mathcal{G})_{c^2 r}$; see \Cref{Lem:Factorization}. We construct families of metric spaces for which this rank is large. In both cases, the metric spaces achieving our lower bounds are the vertex set of a graph equipped with twice the shortest path metric. For the exponential bound, we use the Tur\'an graphs $T(3n,n)$; for the superlinear bound, we use incidence graphs of generalized polygons for $c \in [1, \sqrt{6})$, and Lazebnik--Ustimenko--Woldar graphs $\mathrm{CD}(k,p)$ \cite{lazebnik_ustimenko_woldar_1995} for all $c \geq 1$. We prove the main results in \Cref{Sec:Main-Results}.

\section{Preliminaries}
\label{sec:Preliminaries}

We give the key notation and definitions that we will use throughout the paper.

Let $\Simp$ denote the category of finite abstract simplicial complexes and simplicial maps, which we regard as a subcategory of the category $\Top$ of topological spaces and continuous maps via geometric realization. We regard a poset $P$ as a category in the usual way, with object set $P$ and a morphism $p \to q$ for each $p \leq q$. For any category $\mathbf{C}$, functor $\mathcal{F}\colon P \to \mathbf{C}$, and $p \leq q$ in $P$, we write $\mathcal{F}(p)$ as $\mathcal{F}_p$ and the morphism $\mathcal{F}(p \leq q)\colon \mathcal{F}_p \to \mathcal{F}_q$ as $\mathcal{F}_{p \to q}$. We call the morphisms $\mathcal{F}_{p \to q}$ \emph{structure maps}.

\subsection{Simplicial Complexes and (Bi)filtrations}
\sloppypar

If $P=T_1 \times \cdots \times T_n$ is the Cartesian product of $n$ totally ordered sets, we give $P$ the \emph{product partial order} given by $(t_1, \dots, t_n)\leq (s_1, \dots, s_n)$ iff $t_i \leq s_i$ for all $i$. We will primarily be interested in the case $P=[0, \infty)$ and $P=T \times [0, \infty)$ for some totally ordered set $T$.

A \emph{simplicial filtration} is a functor $\mathcal{F}\colon [0, \infty) \to \Simp$ such that each structure map is an inclusion. A  \emph{bifiltration} is a functor $\mathcal{B}\colon T \times [0,\infty) \to \Simp$ such that all structure maps are inclusions.

We let $(X,\partial)$ be a finite metric space. For $r \geq 0$, the \emph{(Vietoris-)Rips complex} $\rips(X)_r$ is the simplicial complex
\[
\rips(X)_r = \{\sigma \subset X \mid \sigma\neq \emptyset,\ \diam(\sigma) \leq 2r\},
\]
where $\diam(\sigma) = \max_{x,y \in \sigma}\, \partial(x,y)$. Allowing $r$ to vary yields the \emph{Rips filtration} $\rips(X)\colon [0,\infty) \to \Simp$.

For $x\in X$ and $r \geq 0$, we let $B(x,r)$ denote the closed ball of radius $r$ centered at $x$. The \emph{intrinsic \v{C}ech complex} $\mathcal{I}(X)_r$ at scale $r \geq 0$ is the simplicial complex
\[
\mathcal{I}(X)_r = \left\{ \sigma \subset X \;\big|\; \sigma \neq \emptyset,\ \textstyle\bigcap_{x \in \sigma} B(x,r) \neq \emptyset \right\}.
\]
Allowing $r$ to vary yields the \emph{intrinsic \v{C}ech filtration} $\mathcal{I}(X)\colon [0,\infty) \to \Simp$. The stability of $\mathcal I(X)$ was studied in \cite[Section~4.2.2]{chazalPersistenceStability2014}.

A \emph{graph} is a simplicial complex of dimension at most one. Given a graph $G$, the \emph{clique complex} $\Cl(G)$ is the largest simplicial complex with $1$-skeleton $G$; that is, a finite set $\sigma \subset V(G)$ is a simplex of $\Cl(G)$ if and only if every pair of vertices in $\sigma$ is connected by an edge in $G$. Note that $\rips(X)_r = \Cl(\mathcal{G}(X)_r)$, where $\mathcal{G}(X)_r$ is the graph with vertex set $X$ and an edge $[x,y]$ if and only if $\partial(x,y) \leq 2r$. The \emph{girth} of a graph $G$, denoted $g(G)$, is the length of a shortest cycle in $G$; if $G$ is acyclic, we set $g(G) = \infty$.

The \emph{join} of two simplicial complexes $K$ and $L$ on disjoint vertex sets, denoted $K \ast L$, is the simplicial complex with simplices
\[
K \ast L = \{\sigma \cup \tau \mid \sigma \in K,\ \tau \in L\},
\]
where we allow $\sigma$ or $\tau$ to be empty.  The join is associative, and we write $K^{\ast \ell}$ for the $\ell$-fold iterated join of $K$ with itself, defined by $K^{\ast 1} := K$ and $K^{\ast(\ell+1)} := K^{\ast \ell} \ast K$.

One has the following standard K\"{u}nneth-type formula for the reduced homology of simplicial joins. Note that we always consider homology with field coefficients.

\begin{lemma}\label{Lem:Join-Homology}
    Let $K$ and $L$ be simplicial complexes. Then for every $i \geq 0$,
    \[
    \widetilde{H}_{i+1}(K \ast L) \cong \bigoplus_{p+q=i}
    \widetilde{H}_p(K) \otimes \widetilde{H}_q(L).
    \]
\end{lemma}

\begin{proof}
    See, e.g., \cite[Lemma~2.1]{Milnor1956UniversalBundlesII}.
\end{proof}

Given a simplicial complex $K$ on vertex set $V$ with $|V| = n$, the \emph{Alexander dual} of $K$ with respect to $V$ is
\[
K^\vee = \bigl\{S \in 2^V \setminus \{\emptyset, V\} \;\big|\; V \setminus S \notin K\bigr\}.
\]
See \Cref{Fig:Duality} for an example. 

The following result relates the reduced homology of a simplicial complex with that of its Alexander dual---see \cite{bjorner_note_2009} for a short and self-contained proof.

\begin{lemma}[Alexander Duality \cite{bjorner_note_2009}]
\label{Lem:Alexander-Duality}
Let $K$ be a simplicial complex on vertex set $V$ with $|V| = n$. Then for every $i \in \mathbb{Z}$,
\[
\widetilde{H}_i(K^\vee) \cong \widetilde{H}_{n-i-3}(K).
\]
\end{lemma}
\begin{figure}[t]
\centering
\begin{tikzpicture}[
  scale=1.0,
  line cap=round,
  line join=round,
  every node/.style={font=\small}
]

\coordinate (a1) at (0,0);
\coordinate (a2) at (1.8,0);
\coordinate (a3) at (0.9,1.55);
\coordinate (a4) at (2.9,0.78);

\fill[blue!15] (a1)--(a2)--(a3)--cycle;
\draw (a1)--(a2)--(a3)--cycle;

\filldraw[fill=black,draw=black] (a1) circle[radius=1.8pt] node[below left] {$1$};
\filldraw[fill=black,draw=black] (a2) circle[radius=1.8pt] node[below right] {$2$};
\filldraw[fill=black,draw=black] (a3) circle[radius=1.8pt] node[above] {$3$};
\filldraw[fill=black,draw=black] (a4) circle[radius=1.8pt] node[right] {$4$};

\draw[->,thick] (3.8,0.78) -- (5.2,0.78)
  node[midway,above] {$K \mapsto K^\vee$};

\coordinate (b1) at (6.2,0);
\coordinate (b2) at (8.0,0);
\coordinate (b3) at (7.1,1.55);

\draw (b1)--(b2)--(b3)--cycle;

\filldraw[fill=black,draw=black] (b1) circle[radius=1.8pt] node[below left] {$1$};
\filldraw[fill=black,draw=black] (b2) circle[radius=1.8pt] node[below right] {$2$};
\filldraw[fill=black,draw=black] (b3) circle[radius=1.8pt] node[above] {$3$};

\end{tikzpicture}
\caption{A simplicial complex (left) and its Alexander dual (right).}
\label{Fig:Duality}
\end{figure}
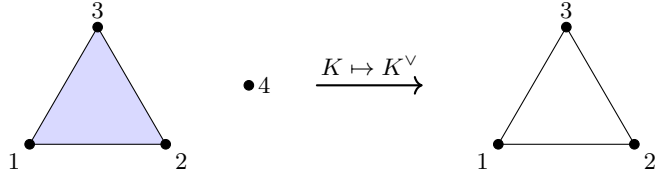

\subsection{Persistence Modules}

Let $\field$ be a fixed field and let $\Vect$ denote the category of $\field$-vector spaces and linear maps. A \emph{persistence module} is a functor $M\colon P \to \Vect$. For $p \in P$, define the \emph{interval module} $\field^{\langle p \rangle}\colon P \to \Vect$ by
\[
\field^{\langle p \rangle}_x =
\begin{cases}
\field & \text{if } p \leq x, \\
0 & \text{otherwise,}
\end{cases}
\qquad
\field^{\langle p \rangle}_{x \to y} =
\begin{cases}
\mathrm{id}_\field & \text{if } p \leq x, \\
0 & \text{otherwise.}
\end{cases}
\]
We say $M$ is \emph{free} if $M \cong \bigoplus_{p \in B} \field^{\langle p \rangle}$ for some multiset $B \subset P$. A \emph{presentation} of $M$ is a morphism of free modules $\partial\colon F_1 \to F_0$ such that $M \cong \operatorname{coker}(\partial)$; we say $M$ is \emph{finitely presented} if $F_0$ and $F_1$ can be chosen to be finitely generated. A presentation is \emph{minimal} if the ranks of $F_0$ and $F_1$ are as small as possible; for $P = T_1 \times \cdots \times T_n$ with each $T_i$ totally ordered, a minimal presentation exists and is unique up to isomorphism \cite{lesnickSparseApproximation2024}. We write $\beta_0(M)$ and $\beta_1(M)$ for the ranks of $F_0$ and $F_1$ in a minimal presentation, called the \emph{zeroth} and \emph{first Betti numbers} of $M$, respectively. We refer the reader to \cite{botnanIntroductionMultiparameter2023} for further details on algebraic aspects of persistence modules.

\subsection{Size of \textbf{SCpx}-valued Functors}

For a simplicial filtration $\mathcal{F}$, each simplex $\sigma$ has a unique birth index, so the total number of simplices in $\bigcup_{t \in T} \mathcal{F}_t$ is a natural and unambiguous notion of the \emph{size} of $\mathcal{F}$. For an arbitrary $\Simp$-valued functor, however, a simplex $\sigma$ may appear and reappear at different indices, making such a simplex count ill-defined. Since we wish to consider approximations by arbitrary functors, as many natural constructions that arise in TDA are not filtrations \cite{deyComputingTopological2014b, kerberBarcodesTowers2019}, we require a notion of size that applies in this level of generality. The following definition, taken from \cite{lesnick2024nerve}, serves this purpose and recovers the simplex count in the filtration case. See \cite[Definition~2.3 and Remark~2.4]{lesnick2024nerve} for further discussion.

Given a functor $\mathcal{F}\colon P \to \Simp$, taking simplicial chains pointwise yields a chain complex
\[
\cdots \to C_2 \mathcal{F} \to C_1 \mathcal{F} \to C_0 \mathcal{F}
\]
of persistence modules $C_j \mathcal{F}\colon P \to \Vect$. We call $\mathcal F$ \emph{finitely presented} if $C_i \mathcal F$ is finitely presented for every $i \geq 0$, and $C_i \mathcal F=0$ for sufficiently large $i$.

For $P = T_1 \times \cdots \times T_n$ with each $T_i$ totally ordered, the \emph{size} of a finitely presented functor $\mathcal{F}\colon P \to \Simp$ is
\[
\beta_1(C_0 \mathcal{F}) + \sum_{j=0}^{\infty} \beta_0(C_j \mathcal{F}).
\]
We denote the size of $\mathcal{F}$ by $|\mathcal{F}|$.

\subsection{Homotopy Interleavings}

A category is said to be \emph{thin} if for any objects $x$ and $y$, there is at most one morphism from $x$ to $y$. For $c \geq 1$, let $I^c$ be the thin category with object set $[0, \infty) \times \{0, 1\}$ and a morphism $(r, i) \to (s, j)$ if and only if either
\begin{enumerate}
    \item[(i)] $rc \leq s$, or
    \item[(ii)] $i = j$ and $r \leq s$.
\end{enumerate}
We then have functors $E_0, E_1\colon [0, \infty) \to I^c$ mapping $r$ to $(r, 0)$ and $(r, 1)$, respectively. For any category $\mathbf{C}$ and functors $\mathcal{F}, \mathcal{F}'\colon [0, \infty) \to \mathbf{C}$, a \emph{$c$-interleaving} between $\mathcal{F}$ and $\mathcal{F}'$ is a functor $Z\colon I^c \to \mathbf{C}$ such that $Z \circ E_0 = \mathcal{F}$ and $Z \circ E_1 = \mathcal{F}'$. If such a $Z$ exists, we say $\mathcal{F}$ and $\mathcal{F}'$ are \emph{$c$-interleaved}.

We extend interleavings to the $2$-parameter case as follows. Let $I^{(1,c)}$ be the thin category with object set $T \times [0, \infty) \times \{0, 1\}$ and a morphism $(t, r, i) \to (t', s, j)$ if and only if either
\begin{enumerate}
    \item[(i)] $(t, cr) \leq (t', s)$, or
    \item[(ii)] $i = j$ and $(t, r) \leq (t', s)$.
\end{enumerate}
We have functors $E_0, E_1\colon T \times [0,\infty) \to I^{(1,c)}$ sending $(t, r)$ to $(t, r, 0)$ and $(t, r, 1)$, respectively. For functors $\mathcal{B}, \mathcal{B}'\colon T \times [0,\infty) \to \mathbf{C}$, a \emph{$c$-interleaving} between $\mathcal{B}$ and $\mathcal{B}'$ is a functor $Z\colon I^{(1,c)} \to \mathbf{C}$ such that $Z \circ E_0 = \mathcal{B}$ and $Z \circ E_1 = \mathcal{B}'$. Note that our definition of $2$-parameter interleaving considers shifts only the $[0, \infty)$ parameter.

When $\mathcal{F}, \mathcal{G}\colon [0, \infty) \to \Simp$ are filtrations and the interleaving maps are inclusions, a $c$-interleaving between $\mathcal{F}$ and $\mathcal{G}$ is equivalent to the conditions
\[
    \mathcal{F}_r \subseteq \mathcal{G}_{cr} \quad \text{and} \quad
    \mathcal{G}_r \subseteq \mathcal{F}_{cr} \quad \text{for all } r \geq 0.
\]

\begin{definition}[Weak Equivalence]\label{Def:Weak-Equivalence}
For functors $\mathcal{F}, \mathcal{F}'\colon P \to \Top$, a natural transformation $\eta\colon \mathcal{F} \to \mathcal{F}'$ is an \emph{objectwise homotopy equivalence} if each component $\eta_p\colon \mathcal{F}_p \to \mathcal{F}'_p$ is a homotopy equivalence. We say $\mathcal{F}$ and $\mathcal{F}'$ are \emph{weakly equivalent}, and write $\mathcal{F} \simeq \mathcal{F}'$, if they are connected by a zigzag of objectwise homotopy equivalences:
\[
\begin{tikzcd}[ampersand replacement=\&,column sep=2ex,row sep=2ex]
   \& \mathcal{W}_1\ar["\simeq",swap]{dl}\ar["\simeq"]{dr}  \&  \&
   \cdots\ar["\simeq",swap]{dl}\ar["\simeq"]{dr}  \&  \&   \mathcal{W}_n
   \ar["\simeq",swap]{dl}\ar["\simeq"]{dr}  \\
\mathcal{F} \&  \& \mathcal{W}_2 \&  \& \mathcal{W}_{n-1} \&  \& \mathcal{F}'.
\end{tikzcd}
\]
\end{definition}

\begin{definition}[Homotopy Interleaving {\cite{blumberg2023universality}}]\label{Def:Homotopy-Interleaving}
For $c \geq 1$, functors $\mathcal{F}, \mathcal{G}$ from $[0,\infty)$ or $T \times [0,\infty)$ to $\Top$ are \emph{$c$-homotopy interleaved} if there exist $c$-interleaved functors $\mathcal{F}', \mathcal{G}'$ with $\mathcal{F} \simeq \mathcal{F}'$ and $\mathcal{G} \simeq \mathcal{G}'$. In this case, we say $\mathcal{G}$ is a \emph{$c$-approximation} to $\mathcal{F}$.
\end{definition}

\section{Main Results}
\label{Sec:Main-Results}

The key tool underlying all of our results is the following lemma, which lower bounds the size of any finitely presented $c$-approximation by the rank of a structure map in the persistent homology module of the target filtration. A version of this argument appears implicitly in the proof of \cite[Corollary~1.5\,(ii)]{lesnick2024nerve}.

\begin{lemma}\label{Lem:Factorization}
Let $c \geq 1$, let $\mathcal{G}\colon [0,\infty) \to \Simp$ be a simplicial filtration, and let $\mathcal{F}\colon [0,\infty) \to \Simp$ be a finitely presented $c$-approximation to $\mathcal{G}$. Then for any $i \geq 0$ and $r \geq 0$,
\[
|\mathcal{F}| \geq \dim H_i(\mathcal{F})_{cr} \geq
\operatorname{rank}\bigl(H_i(\mathcal{G})_r \to H_i(\mathcal{G})_{c^2r}\bigr).
\]
\end{lemma}

\begin{proof}
Since $\mathcal{F}$ is a $c$-approximation to $\mathcal{G}$, there exist functors $\mathcal{F}', \mathcal{G}'\colon [0,\infty) \to \Top$ with $\mathcal{F} \simeq \mathcal{F}'$ and $\mathcal{G} \simeq \mathcal{G}'$, together with a $c$-interleaving $Z\colon I^c \to \Top$ satisfying $Z \circ E_0 = \mathcal{G}'$ and $Z \circ E_1 = \mathcal{F}'$.

For each $r \geq 0$, the morphisms $(r,0) \to (cr,1) \to (c^2r,0)$ in $I^c$ compose to the morphism $(r,0) \to (c^2r,0)$. Applying $H_i \circ Z$ yields a factorization
\[
H_i(\mathcal{G}')_r \longrightarrow H_i(\mathcal{F}')_{cr}
\longrightarrow H_i(\mathcal{G}')_{c^2r}
\]
of the structure map $H_i(\mathcal{G}')_{r \to c^2r}$, so $\operatorname{rank}(H_i(\mathcal{G}')_r \to H_i(\mathcal{G}')_{c^2r}) \leq \dim H_i(\mathcal{F}')_{cr}$. Since $\mathcal{F} \simeq \mathcal{F}'$ and $\mathcal{G} \simeq \mathcal{G}'$, applying $H_i$ yields isomorphisms of persistence modules $H_i(\mathcal{F}) \cong H_i(\mathcal{F}')$ and $H_i(\mathcal{G}) \cong H_i(\mathcal{G}')$, which gives the second inequality.

For the first inequality, $H_i(\mathcal{F})_{cr}$ is a quotient of a subspace of $(C_i\mathcal{F})_{cr}$, so
\[
\dim H_i(\mathcal{F})_{cr} \leq \dim(C_i\mathcal{F})_{cr} \leq
\beta_0(C_i\mathcal{F}) \leq |\mathcal{F}|. \qedhere
\]
\end{proof}

\subsection{Lower Bounds for the Vietoris-Rips Filtration}

We are now ready to present our lower bounds for approximations to the Vietoris-Rips filtration.

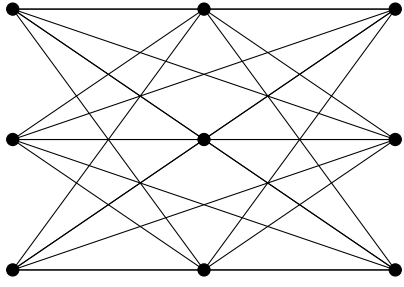
\begin{figure}[t]
\centering
\begin{tikzpicture}[scale=1.15, line cap=round, line join=round]

  \coordinate (a1) at (0,3);
  \coordinate (a2) at (0,1.5);
  \coordinate (a3) at (0,0);

  \coordinate (b1) at (2.2,3);
  \coordinate (b2) at (2.2,1.5);
  \coordinate (b3) at (2.2,0);

  \coordinate (c1) at (4.4,3);
  \coordinate (c2) at (4.4,1.5);
  \coordinate (c3) at (4.4,0);

  \foreach \u in {a1,a2,a3}
    \foreach \v in {b1,b2,b3}
      \draw (\u) -- (\v);

  \foreach \u in {a1,a2,a3}
    \foreach \v in {c1,c2,c3}
      \draw (\u) -- (\v);

  \foreach \u in {b1,b2,b3}
    \foreach \v in {c1,c2,c3}
      \draw (\u) -- (\v);

\foreach \p in {a1,a2,a3,b1,b2,b3,c1,c2,c3}
  \filldraw[fill=black,draw=black] (\p) circle[radius=2pt];

\end{tikzpicture}
\caption{The graph $G_3$.}
\label{Fig:G_3}
\end{figure}

\begin{theorem}\label{Thm:Main-Theorem}
For every $c \in [1, \sqrt{2})$, there exists an infinite family of finite metric spaces $\{X_n\}$ such that any finitely presented $c$-approximation to $\rips(X_n)$ has size at least $2^{|X_n|/3}$.
\end{theorem}

\begin{proof}
For $n \geq 2$, let $G_n$ be the complete $n$-partite graph with $n$ parts of size $3$, and let $X_n$ be the vertex set of $G_n$ equipped with twice the shortest path metric; see \Cref{Fig:G_3} for an example. Note that $|X_n| = 3n$, and that $\rips(X_n)_r = \Cl(G_n)$ for every $r \in [1,2)$.

Let $D_3$ denote the discrete graph on three vertices, viewed as a $0$-dimensional simplicial complex, and write $A_n := D_3^{\ast n}$. One readily checks that $\Cl(G_n) = A_n$. We claim that
\[
\widetilde{H}_i(A_n) \cong
\begin{cases}
\field^{2^n}, & i = n-1, \\
0, & \text{otherwise.}
\end{cases}
\]
We proceed by induction on $n$. The base case $n=1$ holds since $A_1 = D_3$ has $\widetilde{H}_0(D_3) \cong \field^2$ and $\widetilde{H}_i(D_3) = 0$ for $i \geq 1$. Assuming the claim for some $m \geq 1$, we apply \Cref{Lem:Join-Homology} to $A_{m+1} = A_m \ast D_3$ to obtain
\[
\widetilde{H}_{i+1}(A_{m+1}) \cong \bigoplus_{p+q=i} \widetilde{H}_p(A_m) \otimes \widetilde{H}_q(D_3).
\]
By the induction hypothesis, $\widetilde{H}_p(A_m)$ is nonzero only for $p = m-1$, and $\widetilde{H}_q(D_3)$ is nonzero only for $q = 0$. Hence the only nonzero contribution is
\[
\widetilde{H}_m(A_{m+1}) \cong \widetilde{H}_{m-1}(A_m) \otimes \widetilde{H}_0(D_3) \cong \field^{2^m} \otimes \field^2 \cong \field^{2^{m+1}},
\]
and all other reduced homology groups vanish. This establishes the claim.

Since $c^2 \in [1,2)$, we have $\rips(X_n)_1 = \rips(X_n)_{c^2} = A_n$, so the map 
\[
H_{n-1}(\rips(X_n))_1 \to H_{n-1}(\rips(X_n))_{c^2}
\]
 is the identity and has rank $2^n$. By \Cref{Lem:Factorization} applied with $r = 1$ and $i = n-1$, any $c$-approximation $\mathcal{F}$ to $\rips(X_n)$ satisfies $|\mathcal{F}| \geq 2^n = 2^{|X_n|/3}$.
\end{proof}

\begin{remark}\label{Rmk:Turan} 
The graph $G_n$ in the proof of \Cref{Thm:Main-Theorem} is the Turán graph $T(3n, n)$, the complete $n$-partite graph on $3n$ vertices with equal parts of size $3$. The appearance of this graph family is not coincidental: Beers and Botnan \cite{beers_extremal_2025} show that $T(m, i)$ maximizes the $(i-1)$-st Betti number among all flag complexes on $m$ vertices, so $T(3n, n)$ is precisely the flag complex on $3n$ vertices with the largest possible $(n-1)$-st Betti number. When $c = 1$, \Cref{Thm:Main-Theorem} reduces to the statement that exact models of $\rips(-)$ have worst-case exponential size, which was (implicitly) shown by Adamaszek \cite{adamaszek_extremal_2014} using the same family. Tur\'an graphs also appear in Lesnick and McCabe \cite[Corollary~1.5\,(ii)]{lesnick2024nerve}, where they are used to prove exponential lower bounds on the size of $c$-approximations of $\mathcal{SR}(X)$ for $c \in [1, \sqrt 2)$; see also \cite[Remark~4.4]{lesnick2024nerve}. 
\end{remark}

\sloppypar
We now turn to the problem of obtaining lower bounds on the size of $c$-approximations for $\rips(-)$ when $c \geq \sqrt 2$. For $c \in [1, \sqrt 6)$, we can give superlinear lower bounds via incidence graphs of generalized polygons. We give the definitions and properties of these graphs that we will need for \Cref{Cor:Generalized-Polygon-Bounds}, and refer the reader to \cite{van_maldeghem_generalized_1998} for further background.

A \emph{geometry} is a triple $(\mathcal P, \mathcal L, I)$ of \emph{points} $\mathcal P$, \emph{lines} $\mathcal L$, and an \emph{incidence relation} $I \subseteq \mathcal P \times \mathcal L$. We say a point $p \in \mathcal P$ is \emph{incident to} a line $L \in \mathcal L$ if $(p,L) \in I$. The \emph{incidence graph} of $(\mathcal P, \mathcal L, I)$ is the bipartite graph with vertex set $\mathcal P \sqcup \mathcal L$ and an edge $[p, L]$ whenever $p$ is incident to $L$.

A \emph{generalized $n$-gon} is a geometry $\Gamma$ whose incidence graph has diameter $n$, girth $2n$, and minimum vertex degree at least $3$. Every generalized $n$-gon has a well-defined \emph{order $(s, t)$}, where every point is incident to exactly $t+1$ lines and every line to exactly $s+1$ points \cite[Corollary~1.5.3]{van_maldeghem_generalized_1998}; we allow $s = t = \infty$, and call the generalized $n$-gon \emph{finite} when $s, t < \infty$. See \Cref{Fig:Quadrangle} for an example.

We record the existence and sizes of the three families needed for \Cref{Cor:Generalized-Polygon-Bounds}.

\begin{proposition}\label{Prop:n-gons}
\mbox{}
\begin{itemize}
    \item [(1) ]For every prime power $q$, there exists:
\begin{itemize}
    \item [$\bullet$] A generalized quadrangle of order $(q,q)$, with
    \[
    |\mathcal P | = |\mathcal L| = q^3 + q^2 + q + 1.
    \]
    \item [$\bullet$] A generalized hexagon of order $(q,q)$, with
    \[
    |\mathcal P | = |\mathcal L| = q^5 + q^4 + q^3 + q^2 + q + 1
    \]
\end{itemize}
\item [(2)] For every $q$ an odd power of $2$, there exists a generalized octagon of order $(q,q^2)$ with
        \[
        |\mathcal P| = (q+1)(1+q^3)(1+q^6), \qquad |\mathcal L| = (q^2+1)(1+q^3)(1+q^6).
        \]
\end{itemize}
\end{proposition}

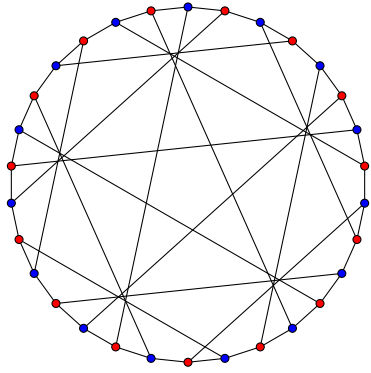
\begin{figure}[t]
\centering
    \begin{tikzpicture}[scale=0.5,line cap=round,line join=round]

\coordinate (v-0)  at ( 0.000,  4.700);
\coordinate (v-1)  at ( 0.977,  4.597);
\coordinate (v-2)  at ( 1.912,  4.294);
\coordinate (v-3)  at ( 2.763,  3.802);
\coordinate (v-4)  at ( 3.493,  3.145);
\coordinate (v-5)  at ( 4.070,  2.350);
\coordinate (v-6)  at ( 4.470,  1.452);
\coordinate (v-7)  at ( 4.674,  0.491);
\coordinate (v-8)  at ( 4.674, -0.491);
\coordinate (v-9)  at ( 4.470, -1.452);
\coordinate (v-10) at ( 4.070, -2.350);
\coordinate (v-11) at ( 3.493, -3.145);
\coordinate (v-12) at ( 2.763, -3.802);
\coordinate (v-13) at ( 1.912, -4.294);
\coordinate (v-14) at ( 0.977, -4.597);
\coordinate (v-15) at ( 0.000, -4.700);
\coordinate (v-16) at (-0.977, -4.597);
\coordinate (v-17) at (-1.912, -4.294);
\coordinate (v-18) at (-2.763, -3.802);
\coordinate (v-19) at (-3.493, -3.145);
\coordinate (v-20) at (-4.070, -2.350);
\coordinate (v-21) at (-4.470, -1.452);
\coordinate (v-22) at (-4.674, -0.491);
\coordinate (v-23) at (-4.674,  0.491);
\coordinate (v-24) at (-4.470,  1.452);
\coordinate (v-25) at (-4.070,  2.350);
\coordinate (v-26) at (-3.493,  3.145);
\coordinate (v-27) at (-2.763,  3.802);
\coordinate (v-28) at (-1.912,  4.294);
\coordinate (v-29) at (-0.977,  4.597);

\draw
(v-0)--(v-1)--(v-2)--(v-3)--(v-4)--(v-5)--(v-6)--(v-7)--(v-8)--(v-9)--
(v-10)--(v-11)--(v-12)--(v-13)--(v-14)--(v-15)--(v-16)--(v-17)--(v-18)--(v-19)--
(v-20)--(v-21)--(v-22)--(v-23)--(v-24)--(v-25)--(v-26)--(v-27)--(v-28)--(v-29)--cycle;

\draw (v-0)--(v-17);
\draw (v-1)--(v-22);
\draw (v-2)--(v-9);
\draw (v-3)--(v-26);
\draw (v-4)--(v-13);
\draw (v-5)--(v-18);
\draw (v-6)--(v-23);
\draw (v-7)--(v-28);
\draw (v-8)--(v-15);
\draw (v-10)--(v-19);
\draw (v-11)--(v-24);
\draw (v-12)--(v-29);
\draw (v-14)--(v-21);
\draw (v-16)--(v-25);
\draw (v-20)--(v-27);

\foreach \i in {0,2,4,6,8,10,12,14,16,18,20,22,24,26,28}{
  \filldraw[fill=blue,draw=black] (v-\i) circle[radius=3pt];
}
\foreach \i in {1,3,5,7,9,11,13,15,17,19,21,23,25,27,29}{
  \filldraw[fill=red,draw=black] (v-\i) circle[radius=3pt];
}

\end{tikzpicture}
\caption{The incidence graph of the generalized quadrangle of order $(2,2)$. The blue vertices correspond to points, while the red vertices correspond to lines.}\label{Fig:Quadrangle}
\end{figure}

\begin{proof}
    See \cite[Chapter~$2$]{van_maldeghem_generalized_1998}.
\end{proof}

The next lemma, which follows directly from \cite[Proposition~$2.2$]{Adamaszek2013}, shows that the Vietoris-Rips filtration of a graph with girth at least $3j+1$ remains topologically stable up to scale $j$.

\begin{lemma}\label{Lem:Collapse}
    Let $G$ be a graph, and let $X := V(G)$ be equipped with twice the shortest path metric. Let $j \geq 1$. If the $1$-skeleton of $\rips(X)_1$, regarded as a graph, has girth at least $3j+1$, then the inclusion $\rips(X)_1 \hookrightarrow \rips(X)_j$ is a homotopy equivalence.
\end{lemma}

\begin{proof}
    Let $H$ be the $1$-skeleton of $\rips(X)_1$. Since $X$ is equipped with twice the shortest path metric on $G$, we have $\rips(X)_1= \Cl(H)$ and also $\rips(X)_j = \Cl (H^j)$, where $H^j$ denotes the $j^{\mathrm{th}}$ graph power of $H$.\footnote{That is, $H^j$ is the graph with vertex set $V(H)$ and an edge between distinct vertices $u,v$ iff their graph distance in $H$ is at most $j$.}
    
    By \cite[Proposition~2.2]{Adamaszek2013}, the assumption $g(H)\geq 3j+1$ implies that the inclusion
\[
\Cl(H)\hookrightarrow \Cl(H^j)
\]
is a homotopy equivalence.
\end{proof}

We now use \Cref{Lem:Collapse} to give superlinear lower bounds on approximations to the Vietoris-Rips filtration. The metric spaces we will use for the lower bounds are the incidence graphs of the three families of generalized polygons in \Cref{Prop:n-gons}.

\begin{theorem}\label{Cor:Generalized-Polygon-Bounds}
For every $c \in [1,\sqrt{6})$, there exists an infinite family of finite metric spaces $\{X_n\}$ such that any finitely presented $c$-approximation to $\rips(X_n)$ has size $\Omega\bigl(|X_n|^{1+\epsilon(c)}\bigr)$, where
\[
\epsilon(c) =
\begin{cases}
1/3, & 1 \leq c < \sqrt{3},\\
1/5, & \sqrt{3} \leq c < 2,\\
1/11, & 2 \leq c < \sqrt{6}.
\end{cases}
\]
\end{theorem}

\begin{proof}
Let $j := \lfloor c^2 \rfloor$, and for each range of $c$, let $G_q$ denote the incidence graph of the appropriate generalized $n$-gon from \Cref{Prop:n-gons}; see table below. 

Set $X_q := V(G_q)$ equipped with twice the shortest path metric. Since $G_q$ is bipartite, it is triangle-free, and therefore $\rips(X_q)_1 = \Cl(G_q) = G_q$. Since the metric is integer-valued, we have $\rips(X_q)_{c^2} = \rips(X_q)_j$, and since $g(G_q) \geq 3j+1$ in each case, \Cref{Lem:Collapse} implies that $\rips(X_q)_1 \hookrightarrow \rips(X_q)_j$ is a homotopy equivalence. Hence the map $H_1(\rips(X_q))_1 \to H_1(\rips(X_q))_{c^2}$ is an isomorphism of rank $\dim H_1(G_q)$. Each of the three cases now follows from the Euler characteristic formula 
\[
\dim H_1(G_q) = |E(G_q)| - |V(G_q)| + 1
\]
and the sizes recorded in \Cref{Prop:n-gons}, together with \Cref{Lem:Factorization} applied with $r = 1$ and $i = 1$:

\medskip
\begin{center}
\begin{tabular}{ccccc}
\hline
range of $c$ & $n$ & $g(G_q)$ & $\dim H_1(G_q)$ & $\epsilon(c)$ \\
\hline
$1 \leq c < \sqrt{3}$ & $4$ & $8$ & $q^4 = \Omega(|X_q|^{4/3})$ & $1/3$ \\
$\sqrt{3} \leq c < 2$  & $6$ & $12$ & $q^6 = \Omega(|X_q|^{6/5})$ & $1/5$ \\
$2 \leq c < \sqrt{6}$  & $8$ & $16$ & $q^{12} = \Omega(|X_q|^{12/11})$ & $1/11$ \\
\hline
\end{tabular}
\end{center}
\medskip

\noindent where $|X_q| = \Theta(q^3)$, $\Theta(q^5)$, and $\Theta(q^{11})$ in the three cases respectively.
\end{proof}

We briefly discuss the potential for extending \Cref{Cor:Generalized-Polygon-Bounds} to larger values of $c$. The strategy of \Cref{Cor:Generalized-Polygon-Bounds} requires a finite generalized $n$-gon with girth $2n \geq 3\lfloor c^2\rfloor + 1$. For $c \geq \sqrt{6}$, this would necessitate $n \geq 10$. However, a classical theorem of Feit and Higman shows that this is impossible.

\begin{theorem}[{\cite[Theorem~1]{feit_nonexistence_1964}; \cite[Section~1.7]{van_maldeghem_generalized_1998}}]
Finite generalized $n$-gons exist only for $n \in \{2, 3, 4, 6, 8\}$.
\end{theorem}

To handle all $c \geq 1$, we turn instead to the Lazebnik--Ustimenko--Woldar graphs $\mathrm{CD}(k,p)$ \cite{lazebnik_ustimenko_woldar_1995}, which have arbitrarily large girth and superlinear edge density, and exist for all positive odd integers $k$ and all primes $p$ (in fact, for all prime powers, but that is not needed here). The graph $\mathrm{CD}(k,p)$ is defined by taking points and lines to be $k$-tuples over the finite field of order $p$, with incidence relation given by a fixed system of bilinear equations on consecutive coordinate pairs. 

For $c \geq 1$, let $k(c)$ be the smallest odd integer with $k > 1$ and $k \geq 3\lfloor c^2 \rfloor - 4$, and define
\begin{equation}\label{Eq:Superlinear}
    d(c) := k(c) - \left\lfloor \tfrac{k(c)+2}{4} \right\rfloor + 1.
\end{equation}

\begin{lemma}\label{Lem:Graph-Construction}
For any $c \geq 1$ and any prime $p$, the graph $G_p := \mathrm{CD}(k(c), p)$ is connected, bipartite, has girth at least $k+5$, and satisfies
\[
g(G_p) \geq 3\lfloor c^2 \rfloor + 1,
\qquad
|E(G_p)| \geq 2^{-1-1/d(c)}\, |V(G_p)|^{1+1/d(c)}.
\]
In particular, $\dim H_1(G_p) = \Omega\bigl(|V(G_p)|^{1+1/d(c)}\bigr)$ as $p \to \infty$.
\end{lemma}

\begin{proof}
By \cite[Theorem~3.2]{lazebnik_ustimenko_woldar_1995}, $\mathrm{CD}(k,p)$ is connected, bipartite, $p$-regular, has girth at least $k+5$, and satisfies $|V(G_p)| \leq 2p^{d(c)}$ and $|E(G_p)| = \frac{1}{2}|V(G_p)|\, p$. Since $k(c) \geq 3\lfloor c^2\rfloor - 4$, we have $g(G_p) \geq k(c) + 5 \geq 3\lfloor c^2\rfloor + 1$. From $|V(G_p)| \leq 2p^{d(c)}$ we get $p \geq (|V(G_p)|/2)^{1/d(c)}$, so
\[
|E(G_p)| = \tfrac{1}{2}|V(G_p)|\,p
\geq 2^{-1-1/d(c)}\,|V(G_p)|^{1+1/d(c)}.
\]
Since $G_p$ is connected, $\dim H_1(G_p) = |E(G_p)| - |V(G_p)| + 1 = \Omega(|V(G_p)|^{1+1/d(c)})$ as $p \to \infty$.
\end{proof}

We are now ready to prove the next theorem.

\begin{theorem}\label{Thm:Superlinear-Bound}
For fixed $c \in [1,\infty)$, there exists a positive constant $\alpha(c)$ and an infinite family of finite metric spaces $\{X_p\}$ such that any finitely presented $c$-approximation to $\rips(X_p)$ has size $\Omega\bigl(|X_p|^{1+\alpha(c)}\bigr)$.
\end{theorem}

\begin{proof}
Let $j := \lfloor c^2 \rfloor$, $\alpha(c) := 1/d(c)$, and $G_p := \mathrm{CD}(k(c), p)$. By \Cref{Lem:Graph-Construction}, $\{G_p\}$ is an infinite family of connected bipartite graphs with $g(G_p) \geq 3j+1$ and $\dim H_1(G_p) = \Omega\bigl(|V(G_p)|^{1+\alpha(c)}\bigr)$.

For each $p$, let $X_p := V(G_p)$ equipped with twice the shortest path metric. Since $G_p$ is bipartite and hence triangle-free, $\rips(X_p)_1 = \Cl(G_p) = G_p$, and since the metric is integer-valued, $\rips(X_p)_{c^2} = \rips(X_p)_j$. Since $g(G_p) \geq 3j+1$, \Cref{Lem:Collapse} implies that $\rips(X_p)_1 \hookrightarrow \rips(X_p)_j$ is a homotopy equivalence, so the map $H_1(\rips(X_p))_1 \to H_1(\rips(X_p))_{c^2}$ is an isomorphism of rank $\Omega\bigl(|X_p|^{1+\alpha(c)}\bigr)$. Applying \Cref{Lem:Factorization} with $r = 1$ and $i = 1$ gives the result.
\end{proof}

\subsection{Lower Bounds for the Intrinsic \v Cech Filtration}
\label{Sec:Cech-Bounds}

In this section, we show that \Cref{Thm:Main-Theorem,Thm:Superlinear-Bound} extend to the intrinsic \v Cech filtration $\mathcal I(-)$. The first extension uses Alexander Duality (cf. \Cref{Lem:Alexander-Duality}), while the second uses the standard interleaving between $\rips(-)$ and $\mathcal I(-)$.

\begin{corollary}\label{Cor:Intrinsic-Cech-Exponential}
For every $c \in [1,\sqrt{2})$, there exists an infinite family of finite metric spaces $\{X_n\}$ such that any finitely presented $c$-approximation to $\mathcal{I}(X_n)$ has size at least $2^{|X_n|/3}$.
\end{corollary}

\begin{proof}
Let $G_n$, $X_n$, and $A_n = D_3^{\ast n}$ be as in the proof of \Cref{Thm:Main-Theorem}, so that $\Cl(G_n) = A_n$ and $\widetilde{H}_{n-1}(A_n) \cong \field^{2^n}$ with all other reduced homology vanishing.

Writing $P_1,\dots,P_n$ for the parts of $G_n$, the closed ball of radius $2$ around any $x \in P_j$ is
\[
B(x,2)=\{x\}\cup (X_n\setminus P_j).
\]
Hence for any nonempty $\sigma \subseteq X_n$,
\[
\bigcap_{x\in \sigma} B(x,2)\neq \emptyset
\]
if and only if there exists $j$ such that $|\sigma\cap P_j|\le 1$: indeed, a vertex $y\in P_j$ lies in $\bigcap_{x\in \sigma} B(x,2)$ exactly when every $x\in \sigma\setminus\{y\}$ lies outside $P_j$, i.e.\ when $\sigma$ contains at most one vertex of $P_j$. Since each part has size $3$, this is equivalent to saying that $X_n\setminus \sigma$ contains at least two vertices from some part, i.e.\ $X_n\setminus \sigma \notin A_n$. Therefore
\[
\mathcal I(X_n)_2 = A_n^\vee.
\] By \Cref{Lem:Alexander-Duality},
\[
\widetilde{H}_i(\mathcal{I}(X_n)_2) \cong \widetilde{H}_{|X_n|-i-3}(A_n),
\]
which, since $|X_n| = 3n$, is $\field^{2^n}$ for $i = 2n-2$ and $0$ otherwise.

Since the metric on $X_n$ is integer-valued and $2c^2 < 4$, we have 
\[
\mathcal{I}(X_n)_2 = \mathcal{I}(X_n)_{2c^2},
\]
 so the structure map
\[
H_{2n-2}(\mathcal{I}(X_n))_2 \to H_{2n-2}(\mathcal{I}(X_n))_{2c^2}
\]
 is the identity of rank $2^n$. Applying \Cref{Lem:Factorization} with $r = 2$ and $i = 2n-2$ gives $|\mathcal{F}| \geq 2^n = 2^{|X_n|/3}$.
\end{proof}

\begin{corollary}\label{Cor:Intrinsic-Cech-Superlinear}
For every fixed $c \in [1,\infty)$, there exists a positive constant $\delta(c)$ and an infinite family of finite metric spaces $\{X_p\}$ such that any finitely presented $c$-approximation to $\mathcal{I}(X_p)$ has size $\Omega\bigl(|X_p|^{1+\delta(c)}\bigr)$.
\end{corollary}

\begin{proof}
Define $\mathcal I'(-): [0, \infty) \to \Simp$ by $\mathcal I'(X)_r := \mathcal I(X)_{\sqrt 2 r}$.

For any finite metric space $X$ and $r \geq 0$, the triangle inequality implies the inclusions
\[
\mathcal{I}(X)_r \subseteq
\rips(X)_r \subseteq \mathcal{I}(X)_{2r}.
\]
Therefore
\[
\mathcal I'(X)_{r/ \sqrt 2} \subseteq \rips(X)_{r} \subseteq \mathcal I'(X)_{\sqrt 2 r},
\]
so $\mathcal I'(X)$ and $\rips(X)$ are $\sqrt 2$-interleaved.

Applying \Cref{Thm:Superlinear-Bound} with $\sqrt2\,c$ in place of $c$, we obtain a positive constant $\delta(c):=\alpha(\sqrt2\,c)$ and a family $\{X_p\}$ such that any finitely presented $\sqrt2\,c$-approximation to $\rips(X_p)$ has size
\[
\Omega\bigl(|X_p|^{1+\delta(c)}\bigr).
\]
If $\mathcal{F}$ is a finitely presented $c$-approximation to $\mathcal{I}(X_p)$, then the linear rescaling $\mathcal F'_r := \mathcal F_{\sqrt 2r}$ is a $c$-approximation to $\mathcal I'(X_p)$, and $|\mathcal F| = |\mathcal F' |$. Composing with the $\sqrt 2$-interleaving above, we get that $\mathcal F'$ is a $\sqrt 2c$-approximation to $\rips(X_p)$, so
\[
| \mathcal F | = |\mathcal{F'}| =
\Omega\bigl(|X_p|^{1+\delta(c)}\bigr). \qedhere
\]
\end{proof}

\subsection{Lower Bounds for Bifiltrations}

Let $\mathcal{F}\colon [0,\infty) \to \Simp$ be a simplicial filtration. We say that a bifiltration $\mathcal{B}$ \emph{contains $\mathcal{F}$ as a slice} if there exists $t_0 \in T$ such that $\mathcal{B}_{t_0,-}\colon [0,\infty) \to \Simp$ is weakly equivalent to $\mathcal{F}$. In this section, we show that \Cref{Thm:Main-Theorem,Thm:Superlinear-Bound} also extend to any bifiltration containing $\rips(-)$ (or, in view of \Cref{Sec:Cech-Bounds}, containing $\mathcal I(-)$) as a slice. 

The following lemma establishes that any approximation to a bifiltration restricts to an approximation to each of its slices, without an increase in size.

\begin{lemma}\label{Lem:Slice}
Let $\mathcal{B}\colon T \times [0,\infty) \to \Simp$ be a bifiltration, and suppose that $\mathcal{F}\colon T \times [0,\infty) \to \Simp$ is a finitely presented $c$-approximation to $\mathcal{B}$. Then for any $t_0 \in T$, the slice $\mathcal{F}_{t_0,-}\colon [0,\infty) \to \Simp$ is a finitely presented $c$-approximation to $\mathcal{B}_{t_0,-}$ satisfying $|\mathcal{F}_{t_0,-}| \leq |\mathcal{F}|$.
\end{lemma}

\begin{proof}
By definition, there exist functors $\mathcal{F}', \mathcal{B}'\colon T \times [0,\infty) \to \Top$ with $\mathcal{F} \simeq \mathcal{F}'$ and $\mathcal{B} \simeq \mathcal{B}'$, together with a $c$-interleaving $Z\colon I^{(1,c)} \to \Top$ satisfying $Z \circ E_0 = \mathcal{B}'$ and $Z \circ E_1 = \mathcal{F}'$. The assignment $(r,i) \mapsto (t_0, r, i)$ induces a functor $I^c \to I^{(1,c)}$, which by composing with $Z$ yields a $c$-interleaving between $\mathcal{B}'_{t_0,-}$ and $\mathcal{F}'_{t_0,-}$. Restricting the zigzags $\mathcal{F} \simeq \mathcal{F}'$ and $\mathcal{B} \simeq \mathcal{B}'$ to the slice $t = t_0$ gives weak equivalences $\mathcal{F}_{t_0,-} \simeq \mathcal{F}'_{t_0,-}$ and $\mathcal{B}_{t_0,-} \simeq \mathcal{B}'_{t_0,-}$, so $\mathcal{F}_{t_0,-}$ is a $c$-approximation to $\mathcal{B}_{t_0,-}$.

For the size bound, let 
\[
P_1 \to P_0 \to C_j\mathcal{F} \to 0
\]
be a minimal presentation of $C_j\mathcal{F}$. Since exactness is defined pointwise, restriction along $\iota_{t_0}: r \mapsto (t_0, r)$ yields a presentation 
\[
P'_1 \to P'_0 \to C_j(\mathcal{F}_{t_0,-}) \to 0.
\]
Since the restriction of a free summand $\field^{\langle(t,r)\rangle}$ on $T \times [0,\infty)$ is either $0$ or $\field^{\langle r\rangle}$, the size of minimal generating sets of $P_1'$ and $P_0'$ are no greater than those of $P_1$ and $P_0$, respectively. Hence $\mathcal{F}_{t_0,-}$ is finitely presented with $\beta_0(C_j(\mathcal{F}_{t_0,-})) \leq \beta_0(C_j\mathcal{F})$ for all $j \geq 0$ and $\beta_1(C_0(\mathcal{F}_{t_0,-})) \leq \beta_1(C_0\mathcal{F})$, giving $|\mathcal{F}_{t_0,-}| \leq |\mathcal{F}|$.
\end{proof}

Combining \Cref{Lem:Slice} with \Cref{Thm:Main-Theorem,Thm:Superlinear-Bound} immediately yields the following two corollaries.  As the proofs are identical in structure, we prove only the first one.

\begin{corollary}\label{Cor:Bifiltration-Exponential}
For every $c \in [1,\sqrt{2})$, there exists an infinite family of finite metric spaces $\{X_n\}$ such that for any bifiltration $\mathcal{B}$ containing $\rips(X_n)$ as a slice, any finitely presented $c$-approximation to $\mathcal{B}$ has size at least $2^{|X_n|/3}$.
\end{corollary}

\begin{proof}
Let $\{X_n\}$ be the family from \Cref{Thm:Main-Theorem}, and let $\mathcal{B}$ be any bifiltration containing $\rips(X_n)$ as a slice, witnessed by some $t_0 \in T$ with $\mathcal{B}_{t_0,-} \simeq \rips(X_n)$. If $\mathcal{F}$ is a finitely presented $c$-approximation to $\mathcal{B}$, then by \Cref{Lem:Slice} the slice $\mathcal{F}_{t_0,-}$ is a $c$-approximation to $\mathcal{B}_{t_0,-} \simeq \rips(X_n)$ with $|\mathcal{F}_{t_0,-}| \leq |\mathcal{F}|$. \Cref{Thm:Main-Theorem} therefore gives $|\mathcal{F}| \geq |\mathcal{F}_{t_0,-}| \geq 2^{|X_n|/3}$.
\end{proof}

\begin{corollary}\label{Cor:Bifiltration-Polynomial}
For fixed $c \in [1,\infty)$, there exists a positive constant $\alpha(c)$ and an infinite family of finite metric spaces $\{X_p\}$ such that for any bifiltration $\mathcal{B}$ containing $\rips(X_p)$ as a slice, any finitely presented $c$-approximation to $\mathcal{B}$ has size $\Omega\bigl(|X_p|^{1+\alpha(c)}\bigr)$.
\end{corollary}

\begin{remark}
Several bifiltrations studied in TDA contain $\rips(X)$ as a slice, so \Cref{Cor:Bifiltration-Exponential,Cor:Bifiltration-Polynomial} apply to them directly. In the \emph{function-Rips bifiltration} \cite{carlssonTheoryMultidimensional2009}, associated to a function $f\colon X \to \mathbb{R}$, the slice at $s = \max_{x \in X} f(x)$ is $\rips(X)_r$. In the \emph{degree-Rips bifiltration} $\mathcal{DR}(X)$ \cite{lesnickInteractiveVisualization2015a}, the slice at degree $k = 0$ is $\rips(X)_r$. For the \emph{subdivision-Rips bifiltration} $\mathcal{SR}(X)$ \cite{sheehy12multicover}, the slice at $k = 1$ is the barycentric subdivision of $\rips(X)_r$, which is naturally homeomorphic to $\rips(X)_r$.
\end{remark}

\bibliographystyle{plainnat}
\bibliography{bibliography}

@article{Adamaszek2013,
  author  = {Micha{\l} Adamaszek},
  title   = {Clique complexes and graph powers},
  journal = {Israel Journal of Mathematics},
  volume  = {196},
  pages   = {295--319},
  year    = {2013},
  doi     = {10.1007/s11856-012-0166-1},
}

@misc{lesnick2024nerve,
  author = {Michael Lesnick and Kenneth McCabe},
  title  = {Nerve models of subdivision bifiltrations},
  year   = {2024},
  note   = {arXiv preprint, \doi{10.48550/arXiv.2406.07679}},
}

@article{blumberg2024stability,
  author  = {Andrew J. Blumberg and Michael Lesnick},
  title   = {Stability of 2-Parameter Persistent Homology},
  journal = {Foundations of Computational Mathematics},
  volume  = {24},
  number  = {2},
  pages   = {385--427},
  year    = {2024},
  doi     = {10.1007/s10208-022-09576-6},
}

@article{lazebnik_ustimenko_woldar_1995,
  author  = {Felix Lazebnik and Vladimir A. Ustimenko and Andrew J. Woldar},
  title   = {A new series of dense graphs of high girth},
  journal = {Bulletin of the American Mathematical Society},
  volume  = {32},
  number  = {1},
  pages   = {73--79},
  year    = {1995},
  doi     = {10.1090/S0273-0979-1995-00569-0},
}

@inproceedings{deyComputingTopological2014b,
  author    = {Tamal K. Dey and Fengtao Fan and Yusu Wang},
  title     = {Computing topological persistence for simplicial maps},
  booktitle = {Proceedings of the 30th Annual Symposium on Computational Geometry},
  pages     = {345--354},
  year      = {2014},
  doi       = {10.1145/2582112.2582165},
}

@article{kerberBarcodesTowers2019,
  author  = {Michael Kerber and Hannah Schreiber},
  title   = {Barcodes of towers and a streaming algorithm for persistent homology},
  journal = {Discrete \& Computational Geometry},
  volume  = {61},
  number  = {4},
  pages   = {852--879},
  year    = {2019},
  doi     = {10.1007/s00454-018-0030-0},
}

@incollection{botnanIntroductionMultiparameter2023,
  author    = {Magnus Bakke Botnan and Michael Lesnick},
  title     = {An introduction to multiparameter persistence},
  booktitle = {Representations of Algebras and Related Structures},
  editor    = {Aslak Bakke Buan and Henning Krause and {\O}yvind Solberg},
  pages     = {77--150},
  publisher = {EMS Press},
  year      = {2023},
  doi       = {10.4171/ecr/19/4},
}

@article{blumberg2023universality,
  author  = {Andrew J. Blumberg and Michael Lesnick},
  title   = {Universality of the homotopy interleaving distance},
  journal = {Transactions of the American Mathematical Society},
  volume  = {376},
  number  = {12},
  pages   = {8269--8307},
  year    = {2023},
  doi     = {10.1090/tran/8738},
}

@article{carlssonTheoryMultidimensional2009,
  author  = {Gunnar Carlsson and Afra Zomorodian},
  title   = {The theory of multidimensional persistence},
  journal = {Discrete \& Computational Geometry},
  volume  = {42},
  number  = {1},
  pages   = {71--93},
  year    = {2009},
  doi     = {10.1007/s00454-009-9176-0},
}

@misc{lesnickInteractiveVisualization2015a,
  author = {Michael Lesnick and Matthew Wright},
  title  = {Interactive visualization of 2-{D} persistence modules},
  year   = {2015},
  note   = {arXiv preprint, \doi{10.48550/arXiv.1512.00180}},
}

@inproceedings{sheehy12multicover,
  author    = {Donald R. Sheehy},
  title     = {A multicover nerve for geometric inference},
  booktitle = {Proceedings of the 24th Canadian Conference on Computational Geometry},
  pages     = {309--314},
  year      = {2012},
  url       = {http://2012.cccg.ca/papers/paper52.pdf},
}

@article{sheehyLinearSizeApproximations2013,
  author  = {Donald R. Sheehy},
  title   = {Linear-size approximations to the {Vietoris}--{Rips} filtration},
  journal = {Discrete \& Computational Geometry},
  volume  = {49},
  number  = {4},
  pages   = {778--796},
  year    = {2013},
  doi     = {10.1007/s00454-013-9513-1},
}

@inproceedings{cavanna15geometric,
  author    = {Nicholas J. Cavanna and Mahmoodreza Jahanseir and Donald R. Sheehy},
  title     = {A geometric perspective on sparse filtrations},
  booktitle = {Proceedings of the 27th Canadian Conference on Computational Geometry},
  pages     = {116--121},
  year      = {2015},
  url       = {https://cccg.ca/proceedings/2015/01.pdf},
}

@article{choudhary_improved_2021,
  author  = {Aruni Choudhary and Michael Kerber and Sharath Raghvendra},
  title   = {Improved approximate {Rips} filtrations with shifted integer lattices and cubical complexes},
  journal = {Journal of Applied and Computational Topology},
  volume  = {5},
  number  = {3},
  pages   = {425--458},
  year    = {2021},
  doi     = {10.1007/s41468-021-00072-4},
}

@article{choudhary_polynomial-sized_2019,
  author  = {Aruni Choudhary and Michael Kerber and Sharath Raghvendra},
  title   = {Polynomial-sized topological approximations using the permutahedron},
  journal = {Discrete \& Computational Geometry},
  volume  = {61},
  number  = {1},
  pages   = {42--80},
  year    = {2019},
  doi     = {10.1007/s00454-017-9951-2},
}

@article{botnan_approximating_2015,
  author  = {Magnus Bakke Botnan and Gard Spreemann},
  title   = {Approximating persistent homology in {Euclidean} space through collapses},
  journal = {Applicable Algebra in Engineering, Communication and Computing},
  volume  = {26},
  number  = {1--2},
  pages   = {73--101},
  year    = {2015},
  doi     = {10.1007/s00200-014-0247-y},
}

@inproceedings{sheehy_sparse_2021,
  author    = {Donald R. Sheehy},
  title     = {A sparse {Delaunay} filtration},
  booktitle = {37th International Symposium on Computational Geometry (SoCG 2021)},
  series    = {Leibniz International Proceedings in Informatics (LIPIcs)},
  volume    = {189},
  pages     = {58:1--58:16},
  year      = {2021},
  doi       = {10.4230/LIPIcs.SoCG.2021.58},
}

@misc{leitao_its_2026,
  author = {Ant{\'o}nio Leit{\~a}o},
  title  = {It's all about covers: Persistent homology of cover refinements},
  year   = {2026},
  note   = {arXiv preprint, \doi{10.48550/arXiv.2602.22784}},
}

@article{edelsbrunner2021multicover,
  author  = {Herbert Edelsbrunner and Georg Osang},
  title   = {The multi-cover persistence of {Euclidean} balls},
  journal = {Discrete \& Computational Geometry},
  volume  = {65},
  number  = {4},
  pages   = {1296--1313},
  year    = {2021},
  doi     = {10.1007/s00454-021-00281-9},
}

@article{corbet2023computing,
  author  = {Ren{\'e} Corbet and Michael Kerber and Michael Lesnick and Georg Osang},
  title   = {Computing the multicover bifiltration},
  journal = {Discrete \& Computational Geometry},
  volume  = {70},
  number  = {2},
  pages   = {376--405},
  year    = {2023},
  doi     = {10.1007/s00454-022-00476-8},
}

@inproceedings{buchetSparseHigher2023,
  author    = {Micka{\"e}l Buchet and Bianca B. Dornelas and Michael Kerber},
  title     = {Sparse higher order {\v{C}}ech filtrations},
  booktitle = {39th International Symposium on Computational Geometry (SoCG 2023)},
  series    = {Leibniz International Proceedings in Informatics (LIPIcs)},
  volume    = {258},
  pages     = {20:1--20:17},
  year      = {2023},
  doi       = {10.4230/LIPIcs.SoCG.2023.20},
}

@misc{lesnickSparseApproximation2024,
  author = {Michael Lesnick and Kenneth McCabe},
  title  = {Sparse approximation of the subdivision-{Rips} bifiltration for doubling metrics},
  year   = {2024},
  note   = {arXiv preprint, \doi{10.48550/arXiv.2408.16716}},
}

@inproceedings{alonso_sparse_2025,
  author    = {{\'A}ngel Javier Alonso},
  title     = {A sparse multicover bifiltration of linear size},
  booktitle = {41st International Symposium on Computational Geometry (SoCG 2025)},
  series    = {Leibniz International Proceedings in Informatics (LIPIcs)},
  volume    = {332},
  pages     = {6:1--6:18},
  year      = {2025},
  doi       = {10.4230/LIPIcs.SoCG.2025.6},
}

@misc{hellmerDensitySensitive2024,
  author = {Niklas Hellmer and Jan Spali{\'n}ski},
  title  = {Density sensitive bifiltered {Dowker} complexes via total weight},
  year   = {2024},
  note   = {arXiv preprint, \doi{10.48550/arXiv.2405.15592}},
}

@article{edelsbrunnerMaximumPersistent2026,
  author  = {Herbert Edelsbrunner and Matthew Kahle and Shu Kanazawa},
  title   = {Maximum persistent {Betti} numbers of {\v{C}}ech complexes},
  journal = {Journal of Applied and Computational Topology},
  volume  = {10},
  number  = {1},
  pages   = {5},
  year    = {2026},
  doi     = {10.1007/s41468-026-00233-3},
}

@book{van_maldeghem_generalized_1998,
  author    = {Hendrik Van Maldeghem},
  title     = {Generalized Polygons},
  publisher = {Springer Basel},
  address   = {Basel},
  year      = {1998},
  doi       = {10.1007/978-3-0348-0271-0},
}

@article{feit_nonexistence_1964,
  author  = {Walter Feit and Graham Higman},
  title   = {The nonexistence of certain generalized polygons},
  journal = {Journal of Algebra},
  volume  = {1},
  number  = {2},
  pages   = {114--131},
  year    = {1964},
  doi     = {10.1016/0021-8693(64)90028-6},
}

@article{brun_sparse_2019,
  author  = {Morten Brun and Nello Blaser},
  title   = {Sparse {Dowker} nerves},
  journal = {Journal of Applied and Computational Topology},
  volume  = {3},
  number  = {1--2},
  pages   = {1--28},
  year    = {2019},
  doi     = {10.1007/s41468-019-00028-9},
}

@article{dey_simba_2019,
  author  = {Tamal K. Dey and Dayu Shi and Yusu Wang},
  title   = {{SimBa}: An efficient tool for approximating {Rips}-filtration persistence via simplicial batch-collapse},
  journal = {ACM Journal of Experimental Algorithmics},
  volume  = {24},
  pages   = {1--16},
  year    = {2019},
  doi     = {10.1145/3284360},
}

@article{adams_lower_2024,
  author  = {Henry Adams and {\v{Z}}iga Virk},
  title   = {Lower bounds on the homology of {Vietoris}--{Rips} complexes of hypercube graphs},
  journal = {Bulletin of the Malaysian Mathematical Sciences Society},
  volume  = {47},
  number  = {3},
  pages   = {72},
  year    = {2024},
  doi     = {10.1007/s40840-024-01663-x},
}

@article{goff_extremal_2011,
  author  = {Michael Goff},
  title   = {Extremal {Betti} numbers of {Vietoris}--{Rips} complexes},
  journal = {Discrete \& Computational Geometry},
  volume  = {46},
  number  = {1},
  pages   = {132--155},
  year    = {2011},
  doi     = {10.1007/s00454-010-9274-z},
}

@inproceedings{beers_extremal_2025,
  author    = {Lies Beers and Magnus Bakke Botnan},
  title     = {Extremal {Betti} numbers and persistence in flag complexes},
  booktitle = {41st International Symposium on Computational Geometry (SoCG 2025)},
  series    = {Leibniz International Proceedings in Informatics (LIPIcs)},
  volume    = {332},
  pages     = {14:1--14:18},
  year      = {2025},
  doi       = {10.4230/LIPIcs.SoCG.2025.14},
}

@misc{beers_extremal_2025_arXiv,
  author = {Lies Beers and Magnus Bakke Botnan},
  title  = {Extremal {Betti} numbers and persistence in flag complexes},
  year   = {2025},
  note   = {arXiv preprint, \doi{10.48550/arXiv.2502.21294}},
}

@article{adamaszek_extremal_2014,
  author  = {Micha{\l} Adamaszek},
  title   = {Extremal problems related to {Betti} numbers of flag complexes},
  journal = {Discrete Applied Mathematics},
  volume  = {173},
  pages   = {8--15},
  year    = {2014},
  doi     = {10.1016/j.dam.2014.04.006},
}

@article{bjorner_note_2009,
  author  = {Anders Bj{\"o}rner and Martin Tancer},
  title   = {Note: Combinatorial {Alexander} duality---a short and elementary proof},
  journal = {Discrete \& Computational Geometry},
  volume  = {42},
  number  = {4},
  pages   = {586--593},
  year    = {2009},
  doi     = {10.1007/s00454-008-9102-x},
}

@article{BuchetChazalOudotSheehy2016,
  author = {Micka{\"e}l Buchet and Fr{\'e}d{\'e}ric Chazal and Steve Y. Oudot and Donald R. Sheehy},
  title = {Efficient and robust persistent homology for measures},
  journal = {Computational Geometry},
  volume = {58},
  pages = {70--96},
  year = {2016},
  doi = {10.1016/j.comgeo.2016.07.001}
}

@inproceedings{ChoudharyKerberRaghvendra2019,
  author = {Aruni Choudhary and Michael Kerber and Sharath Raghvendra},
  title = {Improved topological approximations by digitization},
  booktitle = {Proceedings of the Thirtieth Annual ACM-SIAM Symposium on Discrete Algorithms},
  pages = {2675--2688},
  year = {2019},
  doi = {10.1137/1.9781611975482.166}
}

@article{chazalPersistenceStability2014,
	title = {Persistence stability for geometric complexes},
	volume = {173},
	doi = {10.1007/s10711-013-9937-z},
	language = {en},
	number = {1},
	urldate = {2026-06-17},
	journal = {Geometriae Dedicata},
	author = {Chazal, Frédéric and de Silva, Vin and Oudot, Steve},
	year = {2014}
}

@article{Milnor1956UniversalBundlesII,
  author  = {Milnor, John W.},
  title   = {Construction of Universal Bundles, II},
  journal = {Annals of Mathematics},
  series  = {Second Series},
  volume  = {63},
  number  = {3},
  pages   = {430--436},
  year    = {1956},
  doi     = {10.2307/1970012}
}

\end{document}